\documentclass[a4paper,10pt,twoside]{my_cfc}

\usepackage{amssymb,bm,mathrsfs,bbm,amscd}
\usepackage[tbtags]{amsmath}
\usepackage{lastpage}

\usepackage{amsthm}

\newtheorem{theorem}{Theorem}[section]
\newtheorem{lemma}[theorem]{Lemma}
\newtheorem{corollary}[theorem]{Corollary}

\theoremstyle{definition}
\newtheorem{definition}[theorem]{Definition}
\newtheorem{example}[theorem]{Example}

\theoremstyle{remark}
\newtheorem{remark}[theorem]{Remark}

\begin{document}

\title{Composition Functionals in Fractional\\
Calculus of Variations}


\author{Agnieszka B. Malinowska\\
        \texttt{abmalinowska@ua.pt}\\[4mm]
       {\sl Faculty of Computer Science}\\
       {\sl Bia{\l}ystok University of Technology}\\
       {\sl 15-351 Bia{\l}ystok, Poland}
       \and
       Moulay Rchid Sidi Ammi\\
       \texttt{sidiammi@ua.pt}\\[4mm]
       {\sl Department of Mathematics}\\
       {\sl Moulay Ismail University}\\
       {\sl B.P. 509 Errachidia, Morocco}
       \and
       Delfim F. M. Torres\footnote{Corresponding author.}\\
       \texttt{delfim@ua.pt}\\[4mm]
       {\sl Department of Mathematics}\\
       {\sl University of Aveiro}\\
       {\sl 3810-193 Aveiro, Portugal}}

\date{}

\maketitle


\begin{abstract}
We prove Euler--Lagrange and natural boundary necessary optimality
conditions for fractional problems of the calculus of variations
which are given by a composition of functionals. Our approach uses
the recent notions of Riemann--Liouville fractional derivatives and
integrals in the sense of Jumarie. As an application, we get
optimality conditions for the product and the quotient of
fractional variational functionals.
\end{abstract}




\setcounter{page}{1}


\section{Introduction}

The present work is dedicated to the study of general
(non-classical) fractional problems of calculus of variations.
As a particular case, when $\alpha \rightarrow 1$, one gets the
generalized calculus of variations \cite{CLP} with functionals of the form
\begin{equation*}
H \left(\int_{a}^{b}f(t,x(t),x'(t))dt \right),
\end{equation*}
where $f$ has $n$ components and $H$ has $n$ independent variables.
Problems of calculus of variations as these appear in practical
applications (see \cite{CLP,maltorp,maltord} and the references
given therein) but cannot be solved using the classical theory.
Therefore, an extension of this theory is needed.

The fractional calculus of variations started in 1996 with the work
of Riewe \cite{rie}. Riewe formulated the problem of the calculus of
variations with fractional derivatives and obtained the respective
Euler--Lagrange equations, combining both conservative and
nonconservative cases. Nowadays the fractional calculus of
variations is a subject under strong research. Different definitions
for fractional derivatives and integrals are used, depending on the
purpose under study. Investigations cover problems depending on
Riemann--Liouville fractional derivatives (see, \textrm{e.g.},
\cite{Atanackovic,Frederico:Torres1,Frederico:Torres:NODY,withTatiana:Spain2010}),
the Caputo fractional derivative (see, \textrm{e.g.},
\cite{AGRA,Baleanu1,MalTor}), the symmetric fractional derivative
(see, \textrm{e.g.}, \cite{Klimek,withBasia:Spain2010}),
the Jumarie fractional derivative (see, \textrm{e.g.},
\cite{al:ma:tor,Almeida,Jumarie1,Jumarie2,Jumarie4,Jumarie3,Jumarie5,16}), and others
\cite{Ric:Del:09,Cresson:Gasta:Delfim,El-Nabulsi:Torres:2007,El-Nabulsi:Torres:2008,Frederico:Torres2010}.
For applications of the fractional calculus of variations we refer the
reader to \cite{al:ma:tor,Dreisigmeyer1,Dreisigmeyer2,El-Nabulsi:Torres:2008,Jumarie4,Klimek,Rabei2,Rabei1,Stanislavskya}.
Here we use the fractional calculus proposed by Jumarie. This
modified Riemann--Liouville calculus has shown recently to be very useful in the
fractional calculus of variations for multiple integrals \cite{al:ma:tor},
and provides an efficient tool to solve fractional differential equations \cite{16}.

The paper is organized as follows. In Section~\ref{sec:prm} we present some
preliminaries on the fractional calculus proposed by Jumarie. Our
results are then given in Section~\ref{sec:Euler}. We begin
Section~\ref{sec:Euler} by formulating the general (non-classical)
fractional problem of calculus of variations \eqref{vp}. The problem
is defined via the fractional derivative and the fractional integral
in the sense of Jumarie. We obtain Euler--Lagrange equations
and natural boundary conditions for the general problem (Theorem~\ref{thm:mr}),
which are then applied to the product (Corollary~\ref{cproduct})
and the quotient (Corollary~\ref{cquotient}).
In Section~{\ref{sec:ex}} we provide an example illustrating our results.


\section{Preliminaries}
\label{sec:prm}

For an introduction to the classical fractional calculus we refer
the reader to \cite{Kilbas,Podlubny,samko}. In this section we
briefly review the main notions and results from the recent
fractional calculus proposed by Jumarie
\cite{Jumarie1,Jumarie2,Jumarie4,Jumarie3}.

\begin{definition}
Let $f:[a,b]\to\mathbb R$ be a continuous (but not necessarily
differentiable) function. The Jumarie fractional derivative of $f$
is defined by
\begin{equation*}
f^{(\alpha)}(t)
:=\frac{1}{\Gamma(-\alpha)}\int_0^t(t-\tau)^{-\alpha-1}(f(\tau)-f(a))\,d\tau,
\quad \alpha<0,
\end{equation*}
where $\Gamma(z)=\int_0^\infty t^{z-1}e^{-t}\, dt$.
For a positive $\alpha$,
\begin{equation*}
f^{(\alpha)}(t)=(f^{(\alpha-1)}(t))'
=\frac{1}{\Gamma(1-\alpha)}\frac{d}{dt}\int_0^t(t-\tau)^{-\alpha}(f(\tau)-f(a))\,d\tau,
\end{equation*}
in the case $0<\alpha<1$, and
\begin{equation*}
f^{(\alpha)}(t):=(f^{(\alpha -n)}(t))^{(n)}, \quad n\leq\alpha<n+1,
\quad n\geq 1.
\end{equation*}
\end{definition}

The Jumarie fractional derivative has the following properties:
\begin{itemize}
\item[(i)] The $\alpha$th derivative of a constant is zero.
\item[(ii)] If $0<\alpha\leq 1$, then the Laplace transform
of $f^{(\alpha)}$ is given by
\begin{equation*}
\mathfrak{L}\{f^{(\alpha)}(t)\}=s^{\alpha}\mathfrak{L}\{f(t)\}-s^{\alpha-1}f(0).
\end{equation*}
\item[(iii)]
$(g(t)f(t))^{(\alpha)}=g^{(\alpha)}(t)f(t)+g(t)f^{(\alpha)}(t)$, $0<\alpha<1$.
\end{itemize}

\begin{example}
Let $f(t)=t^\gamma$, $\gamma>0$, and $0<\alpha<1$. Then
$f^{(\alpha)}(x)=\Gamma(\gamma+1)\Gamma^{-1}(\gamma+1-\alpha)t^{\gamma-\alpha}$.
\end{example}

\begin{example}
Let $c$ and $x_0$ be given constants.
The solution of the fractional differential equation
$$
x^{(\alpha)}(t)=c, \quad x(0)=x_0,
$$
is given by
$$
x(t)=\frac{c}{\alpha!}t^{\alpha}+x_0,
$$
where ${\alpha!}:=\Gamma(1+\alpha)$.
\end{example}

The integral with respect to $(dt)^\alpha$ is defined as the
solution of the fractional differential equation
\begin{equation}
\label{int}
dy=f(x)(dx)^{\alpha},\quad x\geq0,\quad y(0)=y_0,
\quad 0<\alpha\leq 1 .
\end{equation}
Such solution is provided by the following result:
\begin{lemma}
\label{integral}
Let $f$ denote a continuous function. The solution of the
equation \eqref{int} is
$$
\int_0^tf(\tau)(d\tau)^\alpha :=
\alpha\int_0^t(t-\tau)^{\alpha-1}f(\tau)d\tau,
\quad 0<\alpha\leq 1.
$$
\end{lemma}

\begin{example}
Let $f(t) \equiv 1$, and $0<\alpha\leq 1$.
Then, $\int_0^t(d\tau)^\alpha=t^{\alpha}$.
\end{example}

\begin{example}
The solution of the fractional differential equation
$x^{(\alpha)}(t)=f(t)$, $x(0)=x_0$, is
$$
x(t)=x_0 +\Gamma^{-1}(\alpha)\int_0^t(t-\tau)^{\alpha-1}f(\tau)d\tau.
$$
\end{example}

In the discussion to follow, we will need the following formula of
integration by parts:
\begin{equation}
\label{int:parts}
\int_a^bu^{(\alpha)}(t)v(t)\, (dt)^\alpha =\alpha! [u(t)v(t)]_a^b
-\int_a^bu(t)v^{(\alpha)}(t)\, (dt)^\alpha,
\end{equation}
where $\alpha!:=\Gamma(1+\alpha)$.


\section{Main Results}
\label{sec:Euler}

The general (non-classical) problem of the fractional calculus of
variations under our consideration consists of extremizing
(\textrm{i.e.}, minimizing or maximizing)
\begin{equation}
\label{vp}
\begin{gathered}
\mathcal{L}[x]=H\left(\int_{a}^{b}f_{1}(t,x(t),x^{(\alpha_1)}(t))(dt)^{\alpha_1},
\ldots, \int_{a}^{b}f_{n}((t,x(t),x^{(\alpha_n)}(t))(dt)^{\alpha_n}\right)\\
(x(a)=x_{a}) \quad (x(b)=x_{b})
\end{gathered}
\end{equation}
over all $x\in \mathbf{D}$ with
$$
\mathbf{D}:=\{x\in C^0 :x^{({\alpha_i})},i=1,\ldots,n,
\mbox{ exists and is continuous on the interval}[a,b]\}.
$$
Using parentheses around the end-point conditions
means that these conditions may or may not be present.
We assume that:
\begin{itemize}

\item[(i)] the function $H:\mathbb{R}^{n}\rightarrow \mathbb{R}$ has continuous
partial derivatives with respect to its arguments and we denote them
by $H'_{i}$, $i=1,\ldots,n$;

\item[(ii)] functions $(t,y,v)\rightarrow f_{i}(t,y,v)$ from $[a,b]\times \mathbb{R}^{2}$ to
$\mathbb{R}$, $i=1,\ldots,n$, have partial continuous derivatives
with respect to $y,v$ for all $t\in[a,b]$ and we denote them by
$f_{iy}$, $f_{iv}$;

\item [(iii)] $f_{i}$, $i=1,\ldots,n$, and their partial
derivatives are continuous in $t$ for all $x\in \mathbf{D} $.
\end{itemize}

A function $x\in \mathbf{D}$ is said to be an admissible function
provided that it satisfies the end-points conditions (if any is
given). The following norm in $\mathbf{D}$ is considered:
\begin{equation*}
\|x\|=\max_{t\in[a,b]}|x(t)|+\sum_{i=1}^n\max_{t\in[a,b]}|x^{(\alpha_i)}(t)|.
\end{equation*}

\begin{definition}
An admissible function $\tilde{x}$ is said to be a
\emph{weak local minimizer} (resp. \emph{weak local maximizer})
for \eqref{vp} if there exists $\delta >0$ such that
$\mathcal{L}[\tilde{x}]\leq \mathcal{L}[x]$
(resp. $\mathcal{L}[\tilde{x}] \geq \mathcal{L}[x]$)
for all admissible $x$ with $\|x-\tilde{x}\|<\delta$.
\end{definition}

For simplicity of notation we introduce the operator
$\langle x \rangle_i$, $i=1,\ldots,n$, defined by
$$
\langle x \rangle_i(t)=(t,x(t),x^{(\alpha_i)}(t)).
$$
Then,
$$
\mathcal{L}[x]=H\left(\int_{a}^{b}f_{1}\langle x \rangle_1(t)(dt)^{\alpha_1},
\ldots, \int_{a}^{b}f_{n}\langle x\rangle_n(t)(dt)^{\alpha_n}\right).
$$

The next theorem gives necessary optimality conditions for problem \eqref{vp}.

\begin{theorem}
\label{thm:mr}
If $\tilde{x}$ is a weak local solution to problem
\eqref{vp}, then the Euler--Lagrange equation
\begin{equation*}
\sum_{i=1}^{n}\alpha
_{i}H'_{i}(\mathcal{F}_{1}[\tilde{x}],\ldots,
\mathcal{F}_{n}[\tilde{x}])(b-t)^{\alpha_i-1}\left(f_{iy}\langle
\tilde{x} \rangle_i(t)- f_{iv}^{(\alpha_i)}\langle \tilde{x}
\rangle_i(t)\right)=0
\end{equation*}
holds for all $t \in [a,b)$, where
$\mathcal{F}_{i}[\tilde{x}]=\int_{a}^{b}f_{i}\langle \tilde{x}
\rangle_i(t)(dt)^{\alpha_i}$, $i=1,\ldots,n$.
Moreover, if $x(a)$ is not specified, then
\begin{equation}
\label{nat:l} \sum_{i=1}^{n}\alpha_{i}!
H'_{i}(\mathcal{F}_{1}[\tilde{x}],\ldots, \mathcal{F}_{n}[\tilde{x}])
f_{iv}\langle \tilde{x} \rangle_i(a)=0 \, ;
\end{equation}
if $x(b)$ is not specified, then
\begin{equation}
\label{nat:r}
\sum_{i=1}^{n}\alpha_{i}! H'_{i}(\mathcal{F}_{1}[\tilde{x}],\ldots,
\mathcal{F}_{n}[\tilde{x}])f_{iv}\langle \tilde{x} \rangle_i(b)=0.
\end{equation}
\end{theorem}

\begin{proof}
Suppose that $\mathcal{L}[x]$ has a weak local extremum at
$\tilde{x}$. For an admissible variation $h\in \mathbf{D}$ we define
a function $\phi:\mathbb{R}\rightarrow \mathbb{R}$ by
$\phi(\varepsilon) = \mathcal{L}[(\tilde{x} + \varepsilon h)] $. We
do not require $h(a)=0$ or $h(b)=0$ in case $x(a)$ or $x(b)$,
respectively, is free (it is possible that both are free). A
necessary condition for $\tilde{x}$ to be an extremizer for
$\mathcal{L}[x]$ is given by
$\phi'(\varepsilon)|_{\varepsilon=0} = 0$.
Using the chain rule to obtain the derivative
of a composed function, we get
\begin{equation*}
\phi'(\varepsilon)|_{\varepsilon=0}
=\sum_{i=1}^{n}H'_{i}(\mathcal{F}_{1}[\tilde{x}],\ldots,\mathcal{F}_{n}[\tilde{x}])
\int_a^b \left[ f_{iy}\langle \tilde{x} \rangle_i(t) h(t) +
f_{iv}\langle \tilde{x} \rangle_i(t)
h^{(\alpha_i)}(t)\right](dt)^{\alpha_i}.
\end{equation*}
Integration by parts (see equation \eqref{int:parts})
of the second term of the integrands, gives
\begin{equation*}
\int_a^b f_{iv}\langle \tilde{x} \rangle_i(t)
h^{(\alpha_i)}(t)(dt)^{\alpha_i} =\left[\alpha_i!f_{iv}\langle
\tilde{x} \rangle_i(t)
h(t)\right]_{t=a}^{t=b}-\int_a^bf_{iv}^{(\alpha_i)}\langle \tilde{x}
\rangle_i(t) h(t)(dt)^{\alpha_i}.
\end{equation*}
The necessary condition $\phi'(\varepsilon)|_{\varepsilon=0} = 0$
can be written as
\begin{multline*}
0 = \sum_{i=1}^{n}H'_{i}(\mathcal{F}_{1}[\tilde{x}],\ldots,\mathcal{F}_{n}[\tilde{x}])
\int_a^b\left(f_{iy}\langle \tilde{x} \rangle_i(t)-f_{iv}^{(\alpha_i)}\langle \tilde{x}
\rangle_i(t)\right)h(t)(dt)^{\alpha_i} \\
+\sum_{i=1}^{n}H'_{i}(\mathcal{F}_{1}[\tilde{x}],\ldots,\mathcal{F}_{n}[\tilde{x}])
\left[\alpha_i!f_{iv}\langle \tilde{x} \rangle_i(t) h(t)\right]_{t=a}^{t=b}.
\end{multline*}
Taking into account Lemma~\ref{integral}, we have
\begin{multline}
\label{eq:aft:IP}
0 = \int_a^b\sum_{i=1}^{n}\alpha_iH'_{i}(\mathcal{F}_{1}[\tilde{x}],
\ldots,\mathcal{F}_{n}[\tilde{x}])(b-t)^{\alpha_i-1}
\left(f_{iy}\langle \tilde{x} \rangle_i(t)
-f_{iv}^{(\alpha_i)}\langle \tilde{x}
\rangle_i(t)\right)h(t) dt\\
+\sum_{i=1}^{n}H'_{i}(\mathcal{F}_{1}[\tilde{x}],\ldots,\mathcal{F}_{n}[\tilde{x}])
\left[\alpha_i!f_{iv}\langle \tilde{x} \rangle_i(t) h(t)\right]_{t=a}^{t=b}.
\end{multline}
In particular, equation \eqref{eq:aft:IP} holds for all variations which are zero at both ends.
For all such $h$'s, the second term in \eqref{eq:aft:IP} is zero and by the
Dubois-Reymond Lemma (see, \textrm{e.g.}, \cite{Brunt}), we have that
\begin{equation}
\label{eq:EL}
\sum_{i=1}^{n}\alpha_{i}H'_{i}(\mathcal{F}_{1}[\tilde{x}],\ldots,
\mathcal{F}_{n}[\tilde{x}])(b-t)^{\alpha_i-1}\left(f_{iy}\langle
\tilde{x} \rangle_i(t)- f_{iv}^{(\alpha_i)}\langle \tilde{x}
\rangle_i(t)\right)=0
\end{equation}
holds for all $t \in [a,b)$. Equation \eqref{eq:aft:IP} must be
satisfied for all admissible values of $h(a)$ and $h(b)$.
Consequently, equations \eqref{eq:aft:IP} and \eqref{eq:EL} imply that
\begin{equation}
\label{eq:1}
0=\sum_{i=1}^{n}H'_{i}(\mathcal{F}_{1}[\tilde{x}],\ldots,\mathcal{F}_{n}[\tilde{x}])
\alpha_i!f_{iv}\langle \tilde{x} \rangle_i(t) h(b)
-\sum_{i=1}^{n}H'_{i}(\mathcal{F}_{1}[\tilde{x}],\ldots,\mathcal{F}_{n}[\tilde{x}])
\alpha_i!f_{iv}\langle \tilde{x} \rangle_i(t) h(a).
\end{equation}
If $x$ is not preassigned at either end-point, then $h(a)$ and
$h(b)$ are both completely arbitrary and we conclude that their
coefficients in \eqref{eq:1} must each vanish. It follows that
condition \eqref{nat:l} holds when $x(a)$ is not given, and
condition \eqref{nat:r} holds when $x(b)$ is not given.
\end{proof}

Note that in the limit, when $\alpha_i\rightarrow 1$,
$i=1,\ldots,n$, Theorem~\ref{thm:mr} implies the following result:

\begin{corollary}[Th.~3.1 and Eq.~(4.1) in \cite{CLP}]
If $\tilde{x}$ is a solution to problem
\begin{equation*}
\begin{gathered}
\mathcal{L}[x]=H\left(\int_{a}^{b}f_{1}(t,x(t),x'(t))dt,\ldots,
\int_{a}^{b}f_{n}(t,x(t),x'(t))dt\right) \longrightarrow \textrm{extr},\\
(x(a)=x_{a}) \quad (x(b)=x_{b})
\end{gathered}
\end{equation*}
then the Euler--Lagrange equation
\begin{equation*}
\sum_{i=1}^{n}H'_{i}(\mathcal{F}_{1}[\tilde{x}],
\ldots,\mathcal{F}_{n}[\tilde{x}])\left(f_{iy}(t,\tilde{x}(t),\tilde{x}'(t))
-\frac{d}{dx}f_{iv}(t,\tilde{x}(t),\tilde{x}'(t))\right)=0
\end{equation*}
holds for all $t \in [a,b]$, where
$\mathcal{F}_{i}[\tilde{x}]=\int_{a}^{b}f_{i}(t,\tilde{x}(t),\tilde{x}'(t))dt$,
$i=1,\ldots,n$. Moreover, if $x(a)$ is not specified, then
\begin{equation*}
\sum_{i=1}^{n}H'_{i}(\mathcal{F}_{1}[\tilde{x}],\ldots,
\mathcal{F}_{n}[\tilde{x}])f_{iv}(a,\tilde{x}(a),\tilde{x}'(a))=0;
\end{equation*}
if $x(b)$ is not specified, then
\begin{equation*}
\sum_{i=1}^{n}H'_{i}(\mathcal{F}_{1}[\tilde{x}],\ldots,
\mathcal{F}_{n}[\tilde{x}])f_{iv}(b,\tilde{x}(b),\tilde{x}'(b))=0.
\end{equation*}
\end{corollary}

\begin{corollary}
\label{cproduct}
If $\tilde{x}$ is a solution to problem
\begin{equation*}
\begin{gathered}
\mathcal{L}[x]=\left(\int_{a}^{b}f_{1}\langle \tilde{x}
\rangle_1(t)(dt)^{\alpha_1}\right)\left(\int_{a}^{b}f_{2}\langle
\tilde{x}\rangle_2(t)(dt)^{\alpha_2}\right) \longrightarrow \textrm{extr},\\
(x(a)=x_{a}) \quad (x(b)=x_{b})
\end{gathered}
\end{equation*}
then the Euler--Lagrange equation
\begin{equation*}
\alpha_1\mathcal{F}_{2}[\tilde{x}](b-t)^{\alpha_1-1}\left(f_{1y}\langle
\tilde{x} \rangle_1(t)- f_{1v}^{(\alpha_1)}\langle \tilde{x}
\rangle_1(t)\right)
+\alpha_2\mathcal{F}_{1}[\tilde{x}](b-t)^{\alpha_2-1}\left(f_{2y}\langle
\tilde{x} \rangle_2(t)- f_{2v}^{(\alpha_2)}\langle \tilde{x}
\rangle_2(t)\right)=0
\end{equation*}
holds for all $t \in [a,b)$. Moreover, if $x(a)$ is not specified, then
\begin{equation*}
\alpha_1!\mathcal{F}_{2}[\tilde{x}]f_{1v}\langle \tilde{x}
\rangle_1(a) +\alpha_2!\mathcal{F}_{1}[\tilde{x}]f_{2v}\langle
\tilde{x} \rangle_2(a)=0;
\end{equation*}
if $x(b)$ is not specified, then
\begin{equation*}
\alpha_1!\mathcal{F}_{2}[\tilde{x}]f_{1v}\langle \tilde{x}
\rangle_1(b) +\alpha_2!\mathcal{F}_{1}[\tilde{x}]f_{2v}\langle
\tilde{x} \rangle_2(b)=0.
\end{equation*}
\end{corollary}

\begin{remark}
In the case $\alpha_i\rightarrow 1$, $i=1,2$,
Corollary~\ref{cproduct} gives a result of \cite{CLP}: the
Euler--Lagrange equation associated with the product functional
\begin{equation*}
\mathcal{L}[x]=\left(\int_{a}^{b}f_{1}(t,x(t),x'(t))dt\right)\left(
\int_{a}^{b}f_{2}(t,x(t),x'(t))dt\right)
\end{equation*}
is
\begin{equation*}
\mathcal{F}_{2}[x]\left(f_{1y}(t,x(t),x'(t))-
\frac{d}{dt}f_{1v}(t,x(t),x'(t))\right)
+\mathcal{F}_{1}[x]\left(f_{2y}(t,x(t),x'(t))-
\frac{d}{dt}f_{2v}(t,x(t),x'(t))\right)=0
\end{equation*}
and the natural condition at $t=a$, when $x(a)$ is free, becomes
\begin{equation*}
\mathcal{F}_{2}[x]f_{1v}(a,x(a),x'(a))+\mathcal{F}_{1}[x]f_{2v}(a,x(a),x'(a))=0.
\end{equation*}
\end{remark}

\begin{corollary}
\label{cquotient}
If $\tilde{x}$ is a solution to problem
\begin{equation*}
\begin{gathered}
\mathcal{L}[x]=\frac{\int_{a}^{b}f_{1}\langle \tilde{x}
\rangle_1(t)(dt)^{\alpha_1}}{\int_{a}^{b}f_{2}\langle \tilde{x}
\rangle_2(t)(dt)^{\alpha_2}} \longrightarrow \textrm{extr},\\
(x(a)=x_{a}) \quad (x(b)=x_{b})
\end{gathered}
\end{equation*}
then the Euler--Lagrange equation
\begin{equation*}
\alpha_1(b-t)^{\alpha_1-1}\left(f_{1y}\langle \tilde{x}
\rangle_1(t)- f_{1v}^{(\alpha_1)}\langle \tilde{x}
\rangle_1(t)\right)
-\alpha_2Q(b-t)^{\alpha_2-1}\left(f_{2y}\langle \tilde{x}
\rangle_2(t)- f_{2v}^{(\alpha_2)}\langle \tilde{x}
\rangle_2(t)\right)=0
\end{equation*}
holds for all $t \in [a,b)$, where
$Q=\frac{\mathcal{F}_{1}[\tilde{x}]}{\mathcal{F}_{2}[\tilde{x}]}$.
Moreover, if $x(a)$ is not specified, then
$\alpha_1!f_{1v}\langle \tilde{x} \rangle_1(a)
-\alpha_2!Qf_{2v}\langle \tilde{x} \rangle_2(a)=0$;
if $x(b)$ is not specified, then
$\alpha_1!f_{1v}\langle \tilde{x} \rangle_1(b)
-\alpha_2!Qf_{2v}\langle \tilde{x} \rangle_2(b)=0$.
\end{corollary}

\begin{remark}
In the case $\alpha_i\rightarrow 1$, $i=1,2$,
Corollary~\ref{cquotient} gives the following result of \cite{CLP}:
the Euler--Lagrange equation associated with the quotient functional
\begin{equation*}
\mathcal{L}[x]
=\frac{\int_{a}^{b}f_{1}(t,x(t),x'(t))dt}{\int_{a}^{b}f_{2}(t,x(t),x'(t))dt}
\end{equation*}
is
\begin{equation*}
f_{1y}(t,x(t),x'(t))-Qf_{2y}(t,x(t),x'(t))
-\frac{d}{dt}\left[(f_{1v}(t,x(t),x'(t))
-Qf_{2v}(t,x(t),x'(t))\right]=0
\end{equation*}
and the natural condition at $t=a$, when $x(a)$ is free, becomes
\begin{equation*}
f_{1v}(a,x(a),x'(a))- Qf_{2v}(a,x(a),x'(a))=0.
\end{equation*}
\end{remark}


\section{An Example}
\label{sec:ex}

Consider the problem
\begin{equation}\label{ex:product}
\begin{gathered}
\text{minimize} \quad
\mathcal{L}[x]=\left(\int_{0}^{1}(x^{(\frac{1}{2})}(t))^2(dt)^{\frac{1}{2}}\right)
\left(\int_{0}^{1}t^{\frac{1}{2}}x^{(\frac{1}{2})}(t)(dt)^{\frac{1}{2}}\right)\\
x(0)=0, \quad x(1)=1.
\end{gathered}
\end{equation}
If $\tilde{x}$ is a local minimizer to \eqref{ex:product}, then the
fractional Euler--Lagrange equation must hold, \textrm{i.e.},
\begin{equation*}
\frac{1}{2}Q_{2}(1-t)^{-\frac{1}{2}}2(\tilde{x}^{(\frac{1}{2})}(t))^{(\frac{1}{2})}
+\frac{1}{2}Q_{1}(1-t)^{-\frac{1}{2}}(t^{\frac{1}{2}})^{(\frac{1}{2})}=0,
\end{equation*}
where
\begin{equation*}
Q_{1}=\int_{0}^{1}(\tilde{x}^{(\frac{1}{2})}(t))^2(dt)^{\frac{1}{2}},
\quad Q_{2}=\int_{0}^{1}t^{\frac{1}{2}}\tilde{x}^{(\frac{1}{2})}(t)
(dt)^{\frac{1}{2}}.
\end{equation*}
Hence,
\begin{equation}
\label{ex:product:euler}
Q_{2}2(\tilde{x}^{(\frac{1}{2})}(t))^{(\frac{1}{2})}+Q_{1}\frac{\sqrt{\pi}}{2}=0.
\end{equation}
If $Q_{2}= 0$, then also $Q_{1}=0$. This contradicts the fact that
a global minimizer to the problem
\begin{equation*}
\begin{gathered}
\text{minimize} \quad
\mathcal{F}_{1}[x]=\int_{0}^{1}(x^{(\frac{1}{2})}(t))^2(dt)^{\frac{1}{2}}\\
x(0)=0, \quad x(1)=1
\end{gathered}
\end{equation*}
is $\bar{x}(t)=t^{\frac{1}{2}}$ and
$\mathcal{F}_{1}[\bar{x}]=(\frac{\sqrt{\pi}}{2})^2$.
This can be easily shown by the results obtained in \cite{Almeida}.
Hence, $Q_{2}\neq 0$ and \eqref{ex:product:euler} implies that
candidate solutions to problem \eqref{ex:product} are those
satisfying the fractional differential equation
\begin{equation}
\label{euler}
(\tilde{x}^{(\frac{1}{2})}(t))^{(\frac{1}{2})}=-\frac{Q_1\sqrt{\pi}}{4Q_2}
\end{equation}
subject to the boundary conditions $x(0)=0$ and $x(1)=1$.
Solving equation \eqref{euler} we obtain
\begin{equation}
\label{sol:P}
x(t)=\frac{1}{\sqrt{\pi}}\int_{0}^{t}\left(\frac{Q_1\pi+4\sqrt{\pi}Q_2}{8Q_2}
-\frac{Q_1}{2Q_2}\tau^{\frac{1}{2}}\right)(t-\tau)^{-\frac{1}{2}}d\tau.
\end{equation}
Substituting \eqref{sol:P} into functionals $\mathcal{F}_1$ and
$\mathcal{F}_2$ gives
\begin{equation}
\label{equation:Q1,Q2}
\begin{cases}
-{\frac {1}{192}}\,{\frac {-32\,{Q_1}^{2}-48\,\pi
\,{Q_2}^{2}+3\,{Q_1}^{2}{ \pi }^{2}}{{Q_2}^{2}}}=Q_1\\
{\frac {1}{96}}\,{\frac {-32\,Q_1+3\,Q_1{\pi }^{2}+12\,{\pi}^{3/2}Q_2}{Q_2}}=Q_2.
\end{cases}
\end{equation}
We obtain the candidate minimizer to problem \eqref{ex:product}
solving the system of equations \eqref{equation:Q1,Q2}:
\begin{equation*}
\tilde{x}(t)=\frac{1}{\sqrt{\pi}}\int_{0}^{t}\left(\frac{Q_1\pi
+4\sqrt{\pi}Q_2}{8Q_2}-\frac{Q_1}{2Q_2}\tau^{\frac{1}{2}}\right)(t-\tau)^{-\frac{1}{2}}d\tau,
\end{equation*}
where
$$
Q_1=\frac{4}{3}\frac{\pi(\sqrt{\pi}(\frac{1}{4}\pi^{\frac{3}{2}}
+\frac{1}{4}\sqrt{\pi^3-8\pi})-4)}{-32+3\pi^2} \  \text{ and } \
Q_2=\frac{1}{12}\pi^{\frac{3}{2}}+\frac{1}{12}\sqrt{\pi^3-8\pi}.
$$


\vskip 15 pt

\leftline{\bf\ Acknowledgments}

\vskip 10 pt

Work supported by the project
\emph{New Explorations in Control Theory Through Advanced Research} (NECTAR)
cofinanced by \emph{Funda\c{c}\~{a}o para a Ci\^{e}ncia e a Tecnologia} (FCT), Portugal,
and the \emph{Centre National de la Recherche Scientifique et Technique} (CNRST), Morocco.
A.B. Malinowska is currently a senior researcher at the University of Aveiro, Portugal,
under the support of BUT, via a project of the Polish Ministry of Science and Higher Education
\emph{Wsparcie Miedzynarodowej Mobilnosci Naukowcow}.

\newpage




\begin{thebibliography}{99}

\vskip 20 pt

\bibitem{CLP}
E. Castillo, A. Luce\~{n}o\ and\ P. Pedregal,
Composition functionals in calculus of variations.
Application to products and quotients,
Math. Models Methods Appl. Sci. 18 (2008), no.~1, 47--75.

\vskip 4 pt

\bibitem{maltorp}
A.B. Malinowska\ and\ D.F.M. Torres,
Euler-Lagrange equations for composition functionals
in calculus of variations on time scales.
In: Proc. Workshop in Control, Nonsmooth Analysis, and Optimization
--- celebrating the 60th birthday of Francis Clarke and Richard
Vinter, Porto, 4-8 May (2009). To appear in
Discrete Contin. Dyn. Syst. Ser. B.
{\tt arXiv:1007.0584}

\vskip 4 pt

\bibitem{maltord}
A.B. Malinowska\ and\ D.F.M. Torres,
A general backwards calculus of variations via duality,
Optim. Lett. (2010), in press.
DOI: 10.1007/s11590-010-0222-x
{\tt arXiv:1007.1679}

\vskip 4 pt

\bibitem{rie}
F. Riewe,
Nonconservative Lagrangian and Hamiltonian mechanics,
Phys. Rev. E (3) 53 (1996), no.~2, 1890--1899.

\vskip 4 pt

\bibitem{Atanackovic}
T.M. Atanackovi\'c, S. Konjik\ and\ S. Pilipovi\'c,
Variational problems with fractional derivatives: Euler-Lagrange equations,
J. Phys. A 41 (2008), no.~9, 095201, 12 pp.

\vskip 4 pt

\bibitem{Frederico:Torres1}
G.S.F. Frederico\ and\ D.F.M. Torres,
A formulation of Noether's theorem for fractional problems
of the calculus of variations,
J. Math. Anal. Appl. 334 (2007), no.~2, 834--846.
{\tt arXiv:math/0701187}

\vskip 4 pt

\bibitem{Frederico:Torres:NODY}
G.S.F. Frederico\ and\ D.F.M. Torres,
Fractional conservation laws in optimal control theory,
Nonlinear Dynam. 53 (2008), no.~3, 215--222.
{\tt arXiv:0711.0609}

\vskip 4 pt

\bibitem{withTatiana:Spain2010}
T. Odzijewicz\ and\ D.F.M. Torres,
Calculus of variations with fractional
and classical derivatives. In: Proc. IFAC Workshop
on Fractional Derivative and Applications, IFAC FDA2010,
University of Extremadura, Badajoz, Spain, October 18-20, 2010.
{\tt arXiv:1007.0567}

\vskip 4 pt

\bibitem{AGRA}
O.P. Agrawal,
Fractional variational calculus and the transversality conditions,
J. Phys. A 39 (2006), no.~33, 10375--10384.

\vskip 4 pt

\bibitem{Baleanu1}
D. Baleanu,
Fractional constrained systems and Caputo derivatives,
J. Comput. Nonlinear Dynam. 3 (2008), no.~2, 199--206.

\vskip 4 pt

\bibitem{MalTor}
A.B. Malinowska\ and\ D.F.M. Torres,
Generalized natural boundary conditions for fractional
variational problems in terms of the Caputo derivative,
Comput. Math. Appl. 59 (2010), no.~9, 3110--3116.
{\tt arXiv:1002.3790}

\vskip 4 pt

\bibitem{Klimek}
M. Klimek,
Lagrangean and Hamiltonian fractional sequential mechanics,
Czechoslovak J. Phys. 52 (2002), no.~11, 1247--1253.

\vskip 4 pt

\bibitem{withBasia:Spain2010}
A.B. Malinowska\ and\ D.F.M. Torres,
Fractional variational calculus in terms
of a combined Caputo derivative. In: Proc. IFAC Workshop
on Fractional Derivative and Applications, IFAC FDA2010,
University of Extremadura, Badajoz, Spain, October 18-20, 2010.
{\tt arXiv:1007.0743}

\vskip 4 pt

\bibitem{al:ma:tor}
R. Almeida, A.B. Malinowska\ and\ D.F.M. Torres,
A fractional calculus of variations for multiple integrals
with application to vibrating string,
J. Math. Phys. 51 (2010), no.~3, 033503, 12 pp.
{\tt arXiv:1001.2722}

\vskip 4 pt

\bibitem{Almeida}
R. Almeida\ and\ D.F.M. Torres,
Fractional variational calculus for nondifferentiable functions,
submitted.

\vskip 4 pt

\bibitem{Jumarie1}
G. Jumarie,
On the representation of fractional Brownian motion
as an integral with respect to $({\rm d}t)\sp a$,
Appl. Math. Lett. 18 (2005), no.~7, 739--748.

\vskip 4 pt

\bibitem{Jumarie2}
G. Jumarie,
Modified Riemann-Liouville derivative and fractional Taylor series
of nondifferentiable functions further results,
Comput. Math. Appl. 51 (2006), no.~9-10, 1367--1376.

\vskip 4 pt

\bibitem{Jumarie4}
G. Jumarie,
Fractional Hamilton-Jacobi equation for the optimal control
of nonrandom fractional dynamics with fractional cost function,
J. Appl. Math. Comput. 23 (2007), no.~1-2, 215--228.

\vskip 4 pt

\bibitem{Jumarie3}
G. Jumarie,
Table of some basic fractional calculus formulae derived from
a modified Riemann-Liouville derivative for non-differentiable functions,
Appl. Math. Lett. 22 (2009), no.~3, 378--385.

\vskip 4 pt

\bibitem{Jumarie5}
G. Jumarie,
Analysis of the equilibrium positions of nonlinear dynamical systems
in the presence of coarse-graining disturbance in space,
J. Appl. Math. Comput. 32 (2010), no.~2, 329--351.

\vskip 4 pt

\bibitem{16}
G.C. Wu\ and\ E.W.M. Lee,
Fractional variational iteration method and its application,
Phys. Lett. A 374 (2010), no.~25, 2506--2509.

\vskip 4 pt

\bibitem{Ric:Del:09}
R. Almeida\ and\ D.F.M. Torres,
Calculus of variations with fractional derivatives and fractional integrals,
Appl. Math. Lett. 22 (2009), no.~12, 1816--1820.
{\tt arXiv:0907.1024}

\vskip 4 pt

\bibitem{Cresson:Gasta:Delfim}
J. Cresson, G.S.F. Frederico\ and\ D.F.M. Torres,
Constants of motion for non-differentiable quantum variational problems,
Topol. Methods Nonlinear Anal. 33 (2009), no.~2, 217--231.
{\tt arXiv:0805.0720}

\vskip 4 pt

\bibitem{El-Nabulsi:Torres:2007}
R.A. El-Nabulsi\ and\ D.F.M. Torres,
Necessary optimality conditions for fractional action-like
integrals of variational calculus with Riemann-Liouville
derivatives of order $(\alpha,\beta)$,
Math. Methods Appl. Sci. 30 (2007), no.~15, 1931--1939.
{\tt arXiv:math-ph/0702099}

\vskip 4 pt

\bibitem{El-Nabulsi:Torres:2008}
R.A. El-Nabulsi\ and\ D.F.M. Torres,
Fractional actionlike variational problems,
J. Math. Phys. 49 (2008), no.~5, 053521, 7pp.
{\tt arXiv:0804.4500}

\vskip 4 pt

\bibitem{Frederico:Torres2010}
G.S.F. Frederico\ and\ D.F.M. Torres,
Fractional Noether's theorem in the Riesz-Caputo sense,
Appl. Math. Comput. (2010), in press.
DOI: 10.1016/j.amc.2010.01.100
{\tt arXiv:1001.4507}

\vskip 4 pt

\bibitem{Dreisigmeyer1}
D.W. Dreisigmeyer\ and\ P.M. Young,
Nonconservative Lagrangian mechanics:
a generalized function approach,
J. Phys. A 36 (2003), no.~30, 8297--8310.

\vskip 4 pt

\bibitem{Dreisigmeyer2}
D.W. Dreisigmeyer\ and\ P.M. Young,
Extending Bauer's corollary to fractional derivatives,
J. Phys. A 37 (2004), no.~11, L117--L121.

\vskip 4 pt

\bibitem{Rabei2}
E.M. Rabei\ and\ B.S. Ababneh,
Hamilton-Jacobi fractional mechanics,
J. Math. Anal. Appl. 344 (2008), no.~2, 799--805.

\vskip 4 pt

\bibitem{Rabei1}
E.M. Rabei, K.I. Nawafleh, R.S. Hijjawi, S.I. Muslih\ and\ D. Baleanu,
The Hamilton formalism with fractional derivatives,
J. Math. Anal. Appl. 327 (2007), no.~2, 891--897.

\vskip 4 pt

\bibitem{Stanislavskya}
A.A. Stanislavsky,
Hamiltonian formalism of fractional systems,
Eur. Phys. J. B Condens. Matter Phys. 49 (2006), no.~1, 93--101.

\vskip 4 pt

\bibitem{Kilbas}
A.A. Kilbas, H.M. Srivastava\ and\ J.J. Trujillo,
{\it Theory and applications of fractional differential equations},
Elsevier, Amsterdam, 2006.

\vskip 4 pt

\bibitem{Podlubny}
I. Podlubny,
{\it Fractional differential equations},
Academic Press, San Diego, CA, 1999.

\vskip 4 pt

\bibitem{samko}
S.G. Samko, A.A. Kilbas\ and\ O.I. Marichev,
{\it Fractional integrals and derivatives},
Gordon and Breach, Yverdon, 1993.

\vskip 4 pt

\bibitem{Brunt}
B. van Brunt,
{\it The calculus of variations},
Springer, New York, 2004.

\vskip 4 pt

\end{thebibliography}
\end{document}